\documentclass [12pt,a4paper,draft]{article}
\usepackage[cp1251]{inputenc}
\usepackage[russian]{babel}
\usepackage{latexsym,amsfonts,amssymb,amsmath,longtable,amsthm}
\usepackage{srcltx}
\usepackage{cite}
\usepackage{mathdots}
\DeclareMathOperator{\GL}{GL} 
\DeclareMathOperator{\SL}{SL} 
 
 \DeclareMathOperator{\PSp}{PSp}

\DeclareMathOperator{\Sp}{Sp}     \DeclareMathOperator{\Sym}{Sym}
 \DeclareMathOperator{\bd}{bd}
\DeclareMathOperator{\diag}{diag} 
 \DeclareMathOperator{\Sl}{Sl}
\newcommand{\F}{\mathbb{F}}       

\newtheorem{lem}{{\bfseries Лемма}}
\newtheorem{theorem}{{\bfseries Теорема}}
\newtheorem{remark}{{\bfseries Замечание}}
\newtheorem{cor}{{\bfseries Следствие}}
\newtheorem{prop}{{\bfseries Предложение}}
\newtheorem{probl}{{\bfseries Проблема}}

\title{\vspace{-1cm} \hfill{\normalsize УДК 512.54}{
\fontfamily{cmr} \fontseries{bx} \selectfont \\ \vspace{1cm} О РАСЩЕПЛЯЕМОСТИ НОРМАЛИЗАТОРА ТОРА В
СИМПЛЕКТИЧЕСКИХ ГРУППАХ}
\thanks{Работа выполнена при финансовой поддержке Научного фонда китайских университетов (номер проекта WK0010000029).}
\date{}
\author{\bf  А.А. Гальт}}

\begin{document}

\sloppy

\maketitle
\pagenumbering{arabic}

\begin{quote}
\noindent{\bf Аннотация.}
В данной работе решен вопрос о расщепляемости нормализатора тора в симплектических группах над конечными полями
и над алгебраически замкнутыми полями положительной характеристики.
\medskip

\noindent{\bf Ключевые слова:} Максимальный тор, нормализатор тора, симплектическая группа, группа Вейля.
 \end{quote}

\section*{Введение}

Конечные группы лиева типа составляют основной массив конечных простых групп. Они возникают из линейных алгебраических групп как множество неподвижных точек эндоморфизма Стейнберга. Пусть $\overline{G}$~--- простая связная линейная алгебраическая группа над алгебраическим замыканием $\overline{\F}_p$ конечного поля положительной характеристики $p$, $\sigma$~--- эндоморфизм Стейнберга, $\overline{T}$~--- максимальный $\sigma$-инвариантный тор в группе  $\overline{G}$. Хорошо известно, что все максимальные торы сопряжены в $\overline{G}$ и факторгруппа $N_{\overline{G}}(\overline{T})/\overline{T}$ изоморфна группе Вейля $W$ группы $\overline{G}$. Возникает естественный вопрос:

\begin{probl}\label{prob1}
Описать группы $\overline{G}$, в которых $N_{\overline{G}}(\overline{T})$ расщепляется над $\overline{T}$.
\end{probl}

\noindent При переходе к конечным группам лиева типа $G$ возникает аналогичный вопрос. Пусть $T=\overline{T}\cap G$~--- максимальный тор группы $G$, $N=N_{\overline{G}}(\overline{T})\cap G$~--- алгебраический нормализатор. Отметим, что $N\leqslant N_G(T)$, но равенство, вообще говоря, может нарушаться.

\begin{probl}\label{prob2}
Описать группы $G$ и их максимальные торы $T$, в которых $N$ расщепляется над~$T$.
\end{probl}

\noindent В данной работе рассматриваются простые связные линейные алгебраические группы $\overline{G}$ типа $C_n$. В этом случае группа $\overline{G}$ либо односвязна и $\overline{G}\simeq\Sp_{2n}(\overline{\F}_p)$, либо имеет присоединенный тип и $\overline{G}\simeq\PSp_{2n}(\overline{\F}_p)$. Ответ на проблему~\ref{prob1} дают следующие два утверждения:

\begin{theorem}\label{th1}
Пусть $\overline{T}$~--- максимальный $\sigma$-инвариантный тор в группе $\overline{G}=\PSp_{2n}(\overline{\F}_p)$ и $\overline{N}=N_{\overline{G}}(\overline{T})$. Тогда тор $\overline{T}$ имеет дополнение в $\overline{N}$ в том и только в том случае, если выполняется одно из следующих условий:
\begin{itemize}
  \item[{\em (1)}] $p=2;$
  \item[{\em (2)}] $n\leqslant2$.
\end{itemize}
\end{theorem}

\begin{cor}\label{cor1.5}
Пусть $\overline{T}$~--- максимальный $\sigma$-инвариантный тор в группе $\overline{G}=\Sp_{2n}(\overline{\F}_p)$ и $\overline{N}=N_{\overline{G}}(\overline{T})$. Тогда тор $\overline{T}$ имеет дополнение в $\overline{N}$ в том и только в том случае, если $p=2$.
\end{cor}

При переходе к конечным группам лиева типа $G$ существует взаимно-однозначное соответствие между классами $G$-сопряженных максимальных $\sigma$-инвариантных торов группы $\overline{G}$ и классами $\sigma$-сопряженности группы Вейля $W$. В случае симплектических групп классы $\sigma$-сопряженности группы $W$ совпадают с обычными классами сопряженности и каждому такому классу соответствует циклический тип $(\overline{n_1})\ldots(\overline{n_k})(n_{k+1})\ldots(n_m)$. Ответ на проблему~\ref{prob2} в случае симплектических групп содержится в следующих двух утверждениях:

\begin{theorem}\label{th2}
Пусть $q$~--- степень числа $p$, $\overline{T}$~--- максимальный $\sigma$-инвариантный тор в группе $\overline{G}=\PSp_{2n}(\overline{\F}_p)$ и
$T=\overline{T}\cap G$~--- соответствующий максимальный тор в группе $G=\PSp_{2n}(q)$, имеющий тип $(\overline{n_1})\ldots(\overline{n_k})(n_{k+1})\ldots(n_m)$. Тогда $T$ имеет дополнение в $N=N_{\overline{G}}(\overline{T})\cap G$ в том и только в том случае, если выполняется одно из следующих условий:
\begin{itemize}
  \item[{\em (1)}] $p=2;$
  \item[{\em (2)}] $m=1;$
  \item[{\em (3)}] $m=2, k=2,$ числа $n_1, n_2$~--- нечетные, $q\equiv3\pmod4;$
  \item[{\em (4)}] $m=2, k=1,$ $n_1$~--- нечетно, $n_2$~--- четно, $q\equiv3\pmod4;$
  \item[{\em (5)}] $m=2, k=0,$ числа $n_1, n_2$~--- четные, $q\equiv3\pmod4;$
  \item[{\em (6)}] $m=2, k=0,$ $q\equiv1\pmod4$.
\end{itemize}
\end{theorem}

\begin{cor}\label{cor2.5}
Пусть $q$~--- степень числа $p$, $\overline{T}$~--- максимальный $\sigma$-инвариантный тор в группе $\overline{G}=\Sp_{2n}(\overline{\F}_p)$ и
$T=\overline{T}\cap G$~--- соответствующий максимальный тор в группе $G=\Sp_{2n}(q)$, имеющий тип $(\overline{n_1})\ldots(\overline{n_k})(n_{k+1})\ldots(n_m)$. Тогда $T$ имеет дополнение в $N=~\!N_{\overline{G}}(\overline{T})\cap G$ в том и только в том случае, если $p=2$.
\end{cor}

Отметим, что ситуация в классических группах, отличных от симплектических, представляется значительно иной. В дальнейшем, планируется исследовать проблемы~\ref{prob1},\ref{prob2} для других классических групп.

\section{Обозначения и используемые результаты}

В работе используются следующие обозначения. Через
\noindent $\{n_1, n_2,\ldots, n_m\}$ всегда обозначается разбиение числа $n$, $p$~--- простое число, $q$~--- некоторая степень числа $p$. Группа всех подстановок на $n$ элементах обозначается $\Sym_n$. Символами  $\varepsilon$, $\varepsilon_i$ всегда обозначены элементы множества $\{+,-\}$.Через $\diag(\lambda_1,\lambda_2,\ldots,\lambda_n)$ обозначается диагональная матрица с элементами $\lambda_1,\lambda_2,\ldots,\lambda_n$ на диагонали. Через $\bd(T_1,T_2,\ldots,T_n)$ обозначается блочно-диагональная матрица с квадратными блоками $T_1,T_2,\ldots,T_n$. Для групп $T$ и $H$ выражение $T\rtimes H$ будет обозначать полупрямое произведение групп $T$ и $H$ c нормальной подгруппой $T$. Пусть $T$ и $N$ подгруппы в группе $G$, $T\trianglelefteq N$. Будем говорить, что {\em $N$ расщепляема над $T$} или {\em $T$ имеет дополнение в $N$}, если существует подгруппа $H$ в $G$, такая что $N=T\rtimes H$.

Основные сведения о линейных алгебраических группах можно найти в~\cite{Hum}. Пусть $\overline{G}$~--- простая связная алгебраическая группа над алгебраическим замыканием $\overline{\F}_p$ конечного поля $\F_p$.
Сюръективный эндоморфизм $\sigma$ группы $\overline{G}$ называется {\em эндоморфизмом Стейнберга} (см. \cite[определение~1.15.1]{GorLySol}), если множество его неподвижных точек $\overline{G}_\sigma$ конечно. Любую группу, удовлетворяющую условию $O^{p'}(\overline{G}_\sigma)\leqslant G\leqslant\overline{G}_\sigma$, будем называть {\em конечной группой лиева типа}. Если $\overline{T}$~--- $\sigma$-инвариантный максимальный тор группы $\overline{G}$, то $T=\overline{T}\cap G$ называется {\em максимальным тором} группы лиева типа $G$. Группу $N_{\overline{G}}(\overline{T})\cap G$ будем обозначать через $N(T,G)$ или просто $N$. Отметим, что справедливо включение $N(T,G)\leqslant N_G(T)$, но равенство, вообще говоря, может нарушаться. Например, если мы рассмотрим $G=\SL_n(2)$, то подгруппа диагональных матриц $T$ группы $G$ тривиальна, значит, $N_G(T)=G$. Но $G=(\SL_n(\overline{\F}_2))_\sigma$, где $\sigma$~--- эндоморфизм Стейнберга $\sigma : (a_{i,j}) \mapsto (a_{i,j}^2)$. Тогда $T=\overline{T}_\sigma$, где $\overline{T}$ является подгруппой диагональных матриц в $\SL_n(\overline{\F}_2)$.
Таким образом, $N(G,T)$ является группой мономиальных матриц в $G$. Поэтому для группы $N(G,T)$ будем использовать термин {\em алгебраический нормализатор}. Именно об алгебраическом нормализаторе говорится в теореме~\ref{th2} и следствии~2. Если $H\unlhd G$, то образ группы $T$ в $G/H$ называется {\em максимальным тором группы} $G/H$.

В $\overline{G}$ всегда есть $\sigma$-инвариантный максимальный тор, который будем обозначать через $\overline{T}$. Все максимальные торы сопряжены c $\overline{T}$ в $\overline{G}$. Через $\overline{N}$ и $W$ будем обозначать нормализатор $N_{\overline{G}}(\overline{T})$ и группу Вейля $\overline{N}/\overline{T}$ соответственно, а через $\pi$~--- естественный гомоморфизм из $\overline{N}$ в $W$. Действие $\sigma$ на $W$ определяется естественным образом. Элементы $w_1, w_2$ называются $\sigma$-сопряженными, если $w_1=(w^{-1})^{\sigma}w_2w$ для некоторого элемента $w$ из $W$.

\begin{prop}{\em\cite[Предложение 3.3.1 и 3.3.3]{Car}}\label{torus}.
Тор $\overline{T}^g$ является $\sigma$-инвариантным тогда и только тогда, когда $g^{\sigma}g^{-1}\in\overline{N}$. Отображение $\overline{T}^g\mapsto\pi(g^{\sigma}g^{-1})$ задает биекцию между классами $G$-сопряженных $\sigma$-инвариантных торов группы $\overline{G}$ и классами $\sigma$-сопряженности группы $W$.
\end{prop}

Как следует, из предложения~\ref{torus} строение тора $(\overline{T}^g)_{\sigma}$ группы $G$ определяется только классом $\sigma$-сопряженности элемента $\pi(g^{\sigma}g^{-1})$. Будем говорить, что тор $T$ группы $G$ {\em имеет тип $w$}, если $T=(\overline{T}^g)_\sigma$ для  некоторого $g\in\overline{G}$ такого, что $\pi(g^{\sigma}g^{-1})=w$.

\begin{prop}{\em\cite[Лемма 1.2]{ButGre}}\label{prop3}.
Пусть $g^{\sigma}g^{-1}\in\overline{N}$ и $\pi(g^{\sigma}g^{-1})=w$. Тогда $(\overline{T}^g)_\sigma=(\overline{T}_{\sigma w})^g$, где $w$ действует на $\overline{T}$ сопряжением.
\end{prop}

\begin{prop}{\em\cite[Предложение 3.3.6]{Car}}\label{normalizer}.
Пусть $g^{\sigma}g^{-1}\in\overline{N}$ и $\pi(g^{\sigma}g^{-1})=w$. Тогда $$(N_{\overline{G}}({\overline{T}}^g))_{\sigma}/({\overline{T}}^g)_{\sigma}\simeq C_{W,\sigma}(w)=\{x\in W| (x^{-1})^{\sigma}wx=w\}.$$
\end{prop}

\section{Симплектические группы}

В данном разделе рассматриваются симплектические группы $\Sp_{2n}(\overline{\F}_p)$, ассоциированные с формой $x_1y_{-1}-x_{-1}y_1+\ldots+x_ny_{-n}-x_{-n}y_n$. Через $\PSp_{2n}(\overline{\F}_p)$ обозначается факторгруппа группы $\Sp_{2n}(\overline{\F}_p)$ по ее центру. Пусть $\overline{G}$~--- простая связная линейная алгебраическая группа типа $C_n$. В этом случае группа $\overline{G}$ либо односвязна и мы будем обозначать ее через $\overline{G}_{sc}$, либо имеет присоединенный тип и будет обозначаться через $\overline{G}_{ad}$. При этом $\overline{G}_{sc}\simeq\Sp_{2n}(\overline{\F}_p)$ и $\overline{G}_{ad}\simeq\PSp_{2n}(\overline{\F}_p)$. Строки и столбцы симплектических матриц размерности $2n$ нумеруются в порядке $1,2,\ldots,n,-1,-2,\ldots,-n$. В качестве максимального $\sigma$-инвариантного тора $\overline{T}_{sc}$  в $\overline{G}_{sc}$ возьмем группу всех матриц вида $\bd(D,D^{-1})$, где $D$~--- невырожденная диагональная матрица размера $n\times n$. Через $\overline{N}_{sc}$ будем обозначать нормализатор тора $\overline{T}_{sc}$ в $\overline{G}_{sc}$, а через $\overline{T}_{ad}$ и $\overline{N}_{ad}$ будем обозначать образ группы $\overline{T}_{sc}$ и $\overline{N}_{sc}$ в $\overline{G}_{ad}$ соответственно.

Действие группы Вейля $W$ на $\overline{T}_{sc}$ реализуется перестановкой элементов на диагонали. Группа $\overline{N}_{sc}$ является подгруппой мономиальных матриц и существует вложение группы $W$ в группу подстановок на множестве $\{1,2,\ldots,n,-1,-2,\ldots,-n\}$. Образ группы $W$ при этом вложении совпадает с группой $\Sl_n$ тех подстановок $\varphi$, для которых справедливо равенство $\varphi(-i)=-\varphi(i)$. Определим следующие элементы из $\Sl_n$:
$$\varphi_1=(1,2)(-1,-2), \varphi_2=(2,3)(-2,-3),\ldots, \varphi_{n-1}=(n-1,n)(-(n-1),-n), \tau=(n,-n).$$
Тогда $\Sl_n=\langle \varphi_1,\varphi_2,\ldots,\varphi_{n-1},\tau \rangle$. Элементы $\varphi_1,\varphi_2,\ldots,\varphi_{n-1},\tau$ соответствуют графу Кокстера типа $C_n$:

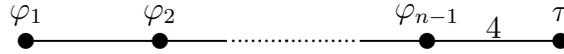
\begin{picture}(30,50)(-60,-20)
\put(50,0){\circle*{6}}
\put(50,0){\line(1,0){50}}
\put(100,0){\circle*{6}}
\put(100,0){\line(1,0){25}}
\put(147.5,0){\circle*{1}} \put(150,0){\circle*{1}}
\put(152.5,0){\circle*{1}} \put(160,0){\circle*{1}}
\put(157.5,0){\circle*{1}} \put(165,0){\circle*{1}}
\put(162.5,0){\circle*{1}} \put(170,0){\circle*{1}}
\put(167.5,0){\circle*{1}} \put(145,0){\circle*{1}}
\put(172.5,0){\circle*{1}} \put(130,0){\circle*{1}}
\put(127.5,0){\circle*{1}} \put(135,0){\circle*{1}}
\put(132.5,0){\circle*{1}} \put(140,0){\circle*{1}}
\put(137.5,0){\circle*{1}} \put(147.5,0){\circle*{1}}
\put(142.5,0){\circle*{1}} \put(155,0){\circle*{1}}
\put(175,0){\line(1,0){25}}
\put(200,0){\circle*{6}}
\put(250,0){\circle*{6}}
\put(200,0){\line(1,0){50}}
\put(50,10){\makebox(0,0){$\varphi_1$}}
\put(100,10){\makebox(0,0){$\varphi_2$}}
\put(250,10){\makebox(0,0){$\tau$}}
\put(200,10){\makebox(0,0){$\varphi_{n-1}$}}
\put(225,5){\makebox(0,0){$4$}}
\end{picture}\\

\noindent В частности, элемент $\tau\varphi_{n-1}$ имеет порядок 4 и $\tau\varphi_i=\varphi_i\tau$ для всех $i\in\{1,\ldots, n-2\}$. Для порождающих элементов группы $\Sl_{n}$ выберем соответствующих представителей из $\overline{N}_{sc}$. Пусть $I_n$~--- единичная матрица размера $n$,\[Q=\left(  \begin{array}{cc}
                0 & I_n \\
                -I_n & 0 \\
           \end{array} \right),\]
тогда $\Sp_{2n}(\overline{\F}_p)=\{A \in \GL_{2n}(\overline{\F}_p) | A^{tr}QA=Q\}$. Поскольку симметрическая группа $\Sym_{2n}$ каноническим образом изоморфна группе подстановочных матриц размерности $2n$, мы будем отождествлять элементы этих групп. Тогда элементы $\varphi_1,\varphi_2,\ldots,\varphi_{n-1}$ лежат в $\overline{N}_{sc}$, а в качестве представителя смежного класса, соответствующего элементу $\tau$, возьмём следующий:

\[ \tau_0=\left(  \begin{array}{c|c|c|c}
                I_{n-1} & 0 & 0 & 0 \\ \hline
                0 & 0 & 0 & 1 \\ \hline
                0 & 0 & I_{n-1} & 0 \\ \hline
                0 & -1 & 0 & 0
           \end{array}
\right).\]
\section{Доказательство теоремы~\ref{th1}}

\begin{remark}\label{rem0}
В случае четной характеристики поля $\overline{\F}_p$ выполняется равенство $\tau=\tau_0$ и группа $\overline{H}_{sc}=\langle \varphi_1,\varphi_2,\ldots,\varphi_{n-1},\tau_0 \rangle$ является дополнением для тора $\overline{T}_{sc}$ в $\overline{N}_{sc}$. Поскольку центр $\Sp_{2n}(\overline{\F}_p)$ тривиален, то $\Sp_{2n}(\overline{\F}_p)=\PSp_{2n}(\overline{\F}_p)$ и группа $\overline{H}_{sc}$ также является дополнением для $\overline{T}_{ad}$ в $\overline{N}_{ad}$.
\end{remark}

\begin{lem}\label{lem-1}
Пусть $\overline{T}_{sc}$~--- максимальный $\sigma$-инвариантный тор в группе $\overline{G}_{sc}=\Sp_{2n}(\overline{\F}_p)$. Если $n\in\{1,2\}$, то тор $\overline{T}_{ad}$ имеет дополнение в $\overline{N}_{ad}$.
\end{lem}
\begin{proof}
Если $n=1$, то положим $\overline{H}_{sc}=\langle\tau_0\rangle$, $\overline{H}_{ad}$~--- образ группы $\overline{H}_{sc}$ в $\PSp_{2}(\overline{\F}_p)$. Тогда $\tau_0^2=-I_2$ лежит в центре $\Sp_{2}(\overline{\F}_p)$ и $\overline{H}_{ad}$ является дополнением для $\overline{T}_{ad}$ в $\overline{N}_{ad}$. В случае $n=2$ группа $W=\langle \varphi_1,\tau\rangle, \varphi_1^2=\tau^2=(\varphi_1\tau)^4=e$. Положим $\overline{H}_{sc}=\langle s_1,t\rangle$, где
\[ s_1=\left(  \begin{array}{cccc}
                0 & 1 & 0 & 0\\
                1 & 0 & 0 & 0 \\
                0 & 0 & 0 & 1 \\
                0 & 0 & 1 & 0 \\
           \end{array} \right)\!,\quad
t=\left(  \begin{array}{cccc}
                \alpha & 0 & 0 & 0\\
                0 & 0 & 0 & 1 \\
                0 & 0 & \alpha^{-1} & 0 \\
                0 & -1 & 0 & 0 \\
           \end{array} \right)\]
и $\alpha\in\overline{\F}_p, \alpha^2=-1$. Тогда $s_1,t\in\Sp_{4}(\overline{\F}_p), s_1^2=I, t^2=(s_1t)^4=-I$. Пусть $\overline{H}_{ad}$~--- образ группы $\overline{H}_{sc}$ в $\PSp_{4}(\overline{\F}_p)$, тогда $\overline{H}_{ad}$ является дополнением для $\overline{T}_{ad}$ в $\overline{N}_{ad}$.
\end{proof}

\begin{lem}\label{lem0}
Пусть $\overline{T}_{sc}$~--- максимальный $\sigma$-инвариантный тор в группе $\overline{G}_{sc}=\Sp_{2n}(\overline{\F}_p)$. Тогда
\begin{itemize}
  \item[{\em (1)}] Если характеристика поля $\overline{\F}_p$ нечетна, то тор $\overline{T}_{sc}$ не имеет дополнения в $\overline{G}_{sc};$
  \item[{\em (2)}] Если $n\geqslant 3$ и характеристика поля $\overline{\F}_p$ нечетна, то тор $\overline{T}_{ad}$ не имеет дополнения в $\overline{N}_{ad}$.
\end{itemize}
\end{lem}
\begin{proof} (1) Предположим противное. Пусть $\overline{H}_{sc}$~--- дополнение для $\overline{T}_{sc}$ в $\overline{N}_{sc}$ и $t$~--- прообраз элемента $\tau$ в $\overline{H}_{sc}$. Тогда $t^2=I$. С другой стороны, элемент $t$ имеет вид
\begin{center}
$t=\diag(\nu_1,\ldots,\nu_{n-1},\nu_n,\nu_1^{-1},\ldots,\nu_{n-1}^{-1},-\nu_n^{-1})\tau_0$,
\end{center}

\noindent для некоторых ненулевых диагональных элементов $\nu_i$. Следовательно,
\[\begin{array}{rcl}
t^2\negthickspace & =\negthickspace & \diag(\nu_1^2,\ldots,\nu_{n-1}^2,-1,\nu_1^{-2},\ldots,\nu_{n-1}^{-2},-1).
\end{array}\]
Противоречие с тем, что $t^2$ единичная матрица. \\
(2) Предположим противное. Пусть $\overline{H}_{ad}$~--- дополнение для $\overline{T}_{ad}$ в $\overline{N}_{ad}$ и $\overline{s}_{n-1},\overline{t}$~--- прообразы элементов $\varphi_{n-1},\tau$ в $\overline{H}_{ad}$. Тогда $(\overline{s}_{n-1})^2=\overline{t}^2=\overline{I}$ и $(\overline{s}_{n-1}\overline{t})^4=\overline{I}$, где $\overline{I}$~--- единичный элемент в $\PSp_{2n}(\overline{\F}_p)$. Пусть $\overline{H}_{sc}$~--- прообраз $\overline{H}_{ad}$ в $\overline{N}_{sc}$ и $s_{n-1},t$~--- прообразы элементов $\overline{s}_{n-1},\overline{t}$ в $\overline{H}_{sc}$. Тогда

\begin{center}
$s_{n-1}=\diag(\mu_1,\mu_2,\ldots,\mu_n,\mu_1^{-1},\mu_2^{-1},\ldots,\mu_n^{-1})\varphi_{n-1}$,

$t=\diag(\nu_1,\nu_2,\ldots,\nu_n,\nu_1^{-1},\nu_2^{-1},\ldots,-\nu_n^{-1})\tau_0$,
\end{center}

\noindent для некоторых ненулевых диагональных элементов $\mu_i,\nu_i$ и матрицы $(s_{n-1})^2,t^2,(s_{n-1}t)^4$ лежат в центре $\Sp_{2n}(\overline{\F}_p)$, в частности, являются скалярными. Непосредственно проверяется, что
\[\begin{array}{rcl}
(s_{n-1})^2\negthickspace & =\negthickspace & \diag(\mu_1^2,\ldots,\mu_{n-2}^2,\mu_{n-1}\mu_n,\mu_{n-1}\mu_n,\mu_1^{-2},
\ldots,\mu_{n-2}^{-2},(\mu_{n-1}\mu_n)^{-1},(\mu_{n-1}\mu_n)^{-1}),\\
t^2\negthickspace & =\negthickspace & \diag(\nu_1^2,\ldots,\nu_{n-2}^2,\nu_{n-1}^2,-1,\nu_1^{-2},\ldots,\nu_{n-2}^{-2},\nu_{n-1}^{-2},-1),
\end{array}\]
откуда, в частности, $\mu_1^2=\mu_1^{-2}$ и $\nu_1^2=\nu_1^{-2}$. Далее,
$$(s_{n-1}t)^4=\diag((\mu_1\nu_1)^4,\ldots,(\mu_{n-2}\nu_{n-2})^4,-1,-1,(\mu_1\nu_1)^{-4},\ldots,(\mu_{n-2}\nu_{n-2})^{-4},-1,-1),$$
откуда $(\mu_1\nu_1)^4=-1$. Противоречие с $\mu_1^4=\nu_1^4=1$.
\end{proof}

\noindent  Теорема~\ref{th1} следует из замечания~\ref{rem0}, леммы~\ref{lem-1} и пункта (2) леммы~\ref{lem0}. Следствие~\ref{cor1.5} следует из замечания~\ref{rem0} и пункта (1) леммы~\ref{lem0}.

\section{Вспомогательные результаты}
Пусть $\sigma$ отображает $\GL_{2n}(\overline{\F}_p)$ в себя по правилу $(a_{ij})\mapsto(a_{ij}^q)$, где $q$~--- степень простого числа $p$. Тогда $G=\overline{G}_{\sigma}=\Sp_{2n}(q)$ и $\widetilde{G}=O^{p'}(\overline{G}_{\sigma})=\PSp_{2n}(q)$. Отметим, что такие обозначения будут удобны при дальнейшем изложении, несмотря на то, что в теореме~\ref{th2} группа $\PSp_{2n}(q)$ обозначена через $G$. Отображение $\sigma$ действует на $W\simeq\Sl_n$ тривиально, поэтому классы $\sigma$-сопряженности совпадают с обычными классами сопряженности. Опустив знаки перед элементами из $\{1,2,\ldots,n,-1,-2,\ldots,-n\}$, получим гомоморфизм из группы $\Sl_n$ на группу $\Sym_n$. Пусть $\varphi\in\Sl_n$ отображается в цикл $(i_1i_2\ldots i_k)$  и оставляет на месте все элементы, отличные от $\pm i_1,\pm i_2,\ldots,\pm i_k$. Если $\varphi(i_k)=i_1$, то назовем $\varphi$ {\em положительным циклом длины $k$}; если $\varphi(i_k)=-i_1$, то назовем $\varphi$ {\em отрицательным циклом длины $k$}. Образ произвольного элемента $\varphi$ из $\Sl_n$ единственным образом раскладывается в произведение независимых циклов, и в соответствии с этим разложением $\varphi$ единственным образом представим в виде произведения независимых положительных и отрицательных циклов. Длины этих циклов вместе с их знаками задают множество целых чисел, которое называется циклическим типом элемента $\varphi$.

Два элемента из $\Sl_n$ сопряжены тогда и только тогда, когда их циклические типы совпадают. Пусть $n=n'+n''$, а $\{n_1,\ldots,n_k\}$ и $\{n_{k+1},\ldots,n_m\}$~--- разбиения чисел $n'$ и $n''$, соответственно. Циклический тип $\{-n_1,\ldots,-n_k,n_{k+1},\ldots,n_m\}$ будет обозначаться через $(\overline{n_1})\ldots(\overline{n_k})(n_{k+1})\ldots(n_m)$. В дальнейшем будут использоваться следующие элементы:\vspace{1em}

\noindent $\sigma_1=(1,2,\ldots,n_1)$,\\
$\sigma_{i+1}=(n_1+\ldots+n_i+1,n_1+\ldots+n_i+2,\ldots,n_1+\ldots+n_i+n_{i+1})$,\\
$\omega_1=(1,2,\ldots,n_1)(-1,-2,\ldots,-n_1)$,\\
$\omega_{i+1}=(n_1+\ldots+n_i+1,\ldots,n_1+\ldots+n_i+n_{i+1})(-(n_1+\ldots+n_i+1),\ldots,-(n_1+\ldots+n_i+n_{i+1}))$,\\
$\varpi_1=(1,2,\ldots,n_1,-1,-2,\ldots,-n_1)$,\\
$\varpi_{i+1}=(n_1+\ldots+n_i+1,\ldots,n_1+\ldots+n_i+n_{i+1},-(n_1+\ldots+n_i+1),\ldots,-(n_1+\ldots+n_i+n_{i+1}))$,\\
$\tau_1=(1,-1)(2,-2)\ldots(n_1,-n_1)$,\\
$\tau_{i+1}=(n_1+\ldots+n_i+1,-(n_1+\ldots+n_i+1))\ldots(n_1+\ldots+n_i+n_{i+1},-(n_1+\ldots+n_i+n_{i+1}))$.\\

\noindent В качестве стандартного представителя типа $(\overline{n_1})\ldots(\overline{n_k})(n_{k+1})\ldots(n_m)$ будем использовать подстановку $\varpi_1\ldots\varpi_k\omega_{k+1}\ldots\omega_{m}$. Строение максимальных торов в группах $\Sp_{2n}(q)$ хорошо известно. Воспользуемся этим описанием из работы \cite[Предолжение 3.1]{ButGre}:

\begin{prop}\label{prop}
Пусть $w$~--- стандартный представитель типа $(\overline{n_1})\ldots(\overline{n_k})(n_{k+1})\ldots(n_m)$. Положим $\varepsilon_i=-$, если $i\leqslant k$, и $\varepsilon_i=+$ в противном случае. Пусть $T$~--- подгруппа в $\Sp_{2n}(\overline{\F}_p)$, состоящая из всех диагональных матриц вида
$$\bd(D_1, D_2,\ldots,D_m, D_1^{-1}, D_2^{-1},\ldots, D_m^{-1}),$$
где $D_i=\diag(\lambda_i, \lambda_i^q,\ldots, \lambda_i^{q^{n_i-1}})$, $\lambda_i^{q^{n_i}-\varepsilon_i1}=1$ для всех $i\in\{1,2,\ldots,m\}$. Тогда $\overline{T}_{\sigma w}\simeq T$.
\end{prop}

\noindent Пусть $N=(N_{\overline{G}}(\overline{T}))_{\sigma w}$. Поскольку $(N_{\overline{G}}(\overline{T}^g))_\sigma=((N_{\overline{G}}(\overline{T}))_{\sigma w})^g$, то в силу предложения~\ref{normalizer} имеем
$$N/T=(N_{\overline{G}}(\overline{T}))_{\sigma w}/\overline{T}_{\sigma w}\simeq(N_{\overline{G}}({\overline{T}}^g))_{\sigma}/({\overline{T}}^g)_{\sigma}\simeq C_{W,\sigma}(w)=C_W(w).$$

\begin{remark}\label{rem01} Как отмечалось в замечании~\ref{rem0}, в случае четной характеристики поля группа $\overline{H}_{sc}\simeq\Sl_n$ содержится в $\overline{N}_{sc}$. Очевидно, что $H=(\overline{H}_{sc})_\sigma=\overline{H}_{sc}\leqslant\Sp_{2n}(q)$. Кроме того, $N/T\simeq C_W(w)$ и $N\leqslant T\rtimes H$. Следовательно, существует подгруппа $H_w\simeq C_W(w)$ в группе $H$, такая что $N=T\rtimes H_w$. Так как центр $\Sp_{2n}(q)$ тривиален, то любой максимальный тор группы $\PSp_{2n}(q)$ также имеет дополнение в своем нормализаторе.
\end{remark}

\noindent Всюду далее рассматривается случай нечетной характеристики поля $\overline{\F}_p$. Элементы $\varpi_1,\ldots,\varpi_k,\omega_{k+1},\ldots,\omega_m,\tau_{k+1},\ldots,\tau_m$ лежат в $C_W(w)$. Выберем представителей для этих элементов в группе $N$. Элементы $\omega_{k+1},\ldots,\omega_m$ принадлежат $\Sp_{2n}(q)$, а значит и $N$. Через $I_j=\diag(1,\ldots,1,1)$ и $C_j=\diag(1,\ldots,1,-1)$ будут обозначаться единичная и диагональная матрицы размера $n_j$ соответственно, а через $I$~--- единичная матрица всей группы. В качестве представителей элементов $\varpi_1,\ldots,\varpi_k$ и $\tau_{k+1},\ldots,\tau_m$ в группе $N$ возьмем следующие:

\begin{center}
$\bd(I_1, \ldots,I_{j-1},I_j,I_{j+1},\ldots,I_m,I_1, \ldots,I_{j-1},C_j,I_{j+1},\ldots,I_m)\varpi_j$\qquad для $1\leqslant j\leqslant k$;
$\bd(I_1, \ldots,I_{j-1},I_j,I_{j+1},\ldots,I_m,I_1, \ldots,I_{j-1},-I_j,I_{j+1},\ldots,I_m)\tau_j$\qquad для $k+1\leqslant j\leqslant m$.
\end{center}

\noindent
Напомним, что мы отождествляем соответствующие элементы группы подстановочных матриц размерности $2n$ и группы подстановок $\Sym_{2n}$. Более того, существует естественное вложение группы $\GL_{n_i}(\overline{\F}_p)$ (соответственно, $\GL_{2n_i}(\overline{\F}_p)$) в группу $\GL_n(\overline{\F}_p)$ (соответственно, $\GL_{2n}(\overline{\F}_p)$) и мы будем отождествлять соответствующие элементы, используя те же обозначения $\sigma_i$ (соответственно, $\omega_i,\varpi_i,\tau_i$).\\

\noindent В дальнейшем нам потребуются следующие леммы~\ref{lem1}--\ref{lem4}.

\begin{lem}\label{lem1}
Пусть $\{n_1, n_2\}$~--- разбиение числа $n$. Пусть $s_1=\bd(T_1, T_2)\sigma_1$, $s_2=\bd(T'_1, T'_2)\sigma_2$, где
$T_1=\diag(\lambda_1, \lambda_2,\ldots, \lambda_{n_1})$, $T_2=\diag(\mu_1, \mu_2,\ldots, \mu_{n_2})$,
$T'_1=\diag(\lambda'_1, \lambda'_2,\ldots, \lambda'_{n_1})$, $T'_2=\diag(\mu'_1, \mu'_2,\ldots, \mu'_{n_2})$. Тогда если $s_1s_2=s_2s_1(zI)$ для некоторого $z\in\overline{\F}_p$, то $\lambda'_{i+1}=\lambda'_iz$, $\mu_j=\mu_{j+1}z$, $\lambda'_1=\lambda'_{n_1}z$, $\mu_{n_2}=\mu_1z$, где $i \in \{1,2,\ldots,n_1-1\}$, $j \in \{1,2,\ldots,n_2-1\}$, а также $z^{n_1}=z^{n_2}=1$.

\begin{proof}
В силу равенства $s_1s_2=s_2s_1(zI)$ имеем $T_1\sigma_1T'_1=T'_1T_1\sigma_1zI_{1}$ и $T_2T'_2\sigma_2=T'_2\sigma_2T_2zI_{2}$. Следовательно,
\begin{center}
$T_1(T'_1)^{\sigma_1^{-1}}=T'_1T_1zI_{1}=T_1T'_1zI_{1}$ и $T'_1=(T'_1)^{\sigma_1}zI_{1}$.
\end{center}
Откуда получаем, что $\lambda'_{i+1}=\lambda'_iz$ для всех $i \in \{1,2,\ldots,n_1-1\}$ и $\lambda'_1=\lambda'_{n_1}z$. Из полученных равенств имеем $\lambda'_1=\lambda'_{n_1}z=\lambda'_{n_1-1}z^2=\ldots=\lambda'_1z^{n_1}$ и $z^{n_1}=1$. Аналогично, получаем
\begin{center}
$T_2T'_2\sigma_2=T'_2\sigma_2T_2zI_{2}$,  $T_2T'_2=T'_2(T_2zI_{2})^{\sigma_2^{-1}}$ и $(T_2)^{\sigma_2}=(T_2)zI_{2}$.
\end{center}
Таким образом, $\mu_j=\mu_{j+1}z$ для всех $j \in \{1,2,\ldots,n_2-1\}$, $\mu_{n_2}=\mu_1z$ и как следствие полученных равенств имеем $z^{n_2}=1$.
\end{proof}
\end{lem}

\begin{lem}\label{lem2}
Пусть $\{n_1, n_2\}$~--- разбиение числа $n$, $q$~--- нечетно. Пусть \\ $t_1=\bd(T_1, D_2, T_3,D_2^{-1})\varpi_1$, $t_2=\bd(D'_1, T'_2, (D'_1)^{-1}, T'_4)\varpi_2$, где $T_1, T'_2, T_3, T'_4$~--- произвольные невырожденные матрицы, $D_2=\diag(\mu_1, \mu_1^q, \ldots, \mu_1^{q^{n_2-1}})$, $D'_1=\diag(\lambda_2, \lambda_2^q, \ldots, \lambda_2^{q^{n_1-1}})$. Тогда
\begin{itemize}
  \item[{\em (1)}] Если $t_1t_2=t_2t_1$, то $\mu_1^2=\lambda_2^2=1;$
  \item[{\em (2)}] Если $t_1t_2=-t_2t_1$, то $\mu_1^2=\lambda_2^2=-1$ и $n_1, n_2$~--- нечетны. Более того, $\mu_1^{q-1}=-1$, если $n_2>1$ и $\lambda_2^{q-1}=-1$, если $n_1>1$.
\end{itemize}
\end{lem}
\begin{proof}
Пусть $D=\bd(D_2,(D_2)^{-1})$, $D'=\bd(D'_1,(D'_1)^{-1})$.

\noindent (1) Аналогично доказательству леммы~\ref{lem1} равенство $t_1t_2=t_2t_1$ равносильно двум равенствам $D^{\varpi_2}=D$ и $(D')^{\varpi_1}=(D')$. Из первого равенства следует, что все диагональные элементы матриц $D_2$ и $D_2^{-1}$ совпадают. В частности, $\mu_1=\mu_1^{-1}$, откуда $\mu_1^2=1$. Аналогично из второго равенства получаем $\lambda_2^2=1$.

\noindent (2) В данном случае равенство $t_1t_2=-t_2t_1$ равносильно двум равенствам $D^{\varpi_2}=-D$ и $(D')^{\varpi_1}=-(D')$. Из первого равенства следует:
\[\left\{\begin{array}{rcl}
\mu_1 & = & -\mu_1^{-q^{n_2-1}} \\
\mu_1^q & = & -\mu_1 \\
& \vdots &  \\
\mu_1^{q^{n_2-1}} & = & -\mu_1^{q^{n_2-2}} \\
\mu_1^{-1} & = & -\mu_1^{q^{n_2-1}} \\
\mu_1^{-q} & = & -\mu_1^{-1} \\
& \vdots &  \\
\mu_1^{-q^{n_2-1}} & = & -\mu_1^{-q^{n_2-2}}
\end{array}
\right.\]
Следовательно, при $n_2>1$ имеем $\mu_1^{q-1}=-1$ и $\mu_1^{-1}=-\mu_1^{q^{n_2-1}}=(-1)^2\mu_1^{q^{n_2-2}}=\ldots=(-1)^{n_2}\mu_1$, откуда $\mu_1^2=(-1)^{n_2}$. Так как $(q-1)$~--- четно, то $n_2$ должно быть нечетным и $\mu_1^2=-1$. В случае $n_2=1$ получаем, что $\mu_1^{-1}=-\mu_1$, откуда $\mu_1^2=-1$. Аналогично, из второго равенства получаем, что $n_1$ также должно быть нечетным, $\lambda_2^2=-1$ и если $n_1>1$, то $\lambda_2^{q-1}=-1$.
\end{proof}

\begin{remark}\label{rem1}
Заключение леммы~\ref{lem2} также справедливо для блочно-диагональных матриц с количеством блоков больше двух.
\end{remark}

\begin{cor}\label{cor1}
Пусть $\{n_1, n_2\}$~--- разбиение числа $n$, $u_1=\bd(T_1, D_2, T_3,D_2^{-1})\tau_1$, $u_2=\bd(D'_1, T'_2, (D'_1)^{-1}, T'_4)\tau_2$,
где $T_1, T'_2, T_3, T'_4$~--- произвольные невырожденные матрицы, $D_2=\diag(\mu_1, \mu_1^q, \ldots, \mu_1^{q^{n_2-1}})$, $D'_1=\diag(\lambda_2, \lambda_2^q, \ldots, \lambda_2^{q^{n_1-1}})$. Если $u_1u_2=u_2u_1$, то $\mu_1^2=\lambda_2^2=1$.
\end{cor}
\begin{proof}
Пусть $D=\bd(D_2,(D_2)^{-1})$, $D'=\bd(D'_1,(D'_1)^{-1})$. Аналогично доказательству леммы~\ref{lem2} равенство $t_1t_2=t_2t_1$ равносильно двум равенствам $D^{\tau_2}=D$ и $(D')^{\tau_1}=(D')$, что в свою очередь равносильно $D_2^{-1}=D_2$ и $(D'_1)^{-1}=D'_1$. Следовательно, $\mu_1=\mu_1^{-1}$ и $\lambda_2=\lambda_2^{-1}$, откуда $\mu_1^2=\lambda_2^2=1$.
\end{proof}

\begin{remark}\label{rem2}
Заключение следствия~\ref{cor1} также справедливо для блочно-диагональных матриц с количеством блоков больше двух.
\end{remark}

\begin{lem}\label{lem2.5}
Пусть $\{n_1, n_2\}$~--- разбиение числа $n$, $q$~--- нечетно. Пусть \\ $t_1=\bd(T_1, D_2, T_3,D_2^{-1})\varpi_1$, $u_2=\bd(D'_1, T'_2, (D'_1)^{-1}, T'_4)\tau_2$, где $T_1, T'_2, T_3, T'_4$~--- произвольные невырожденные матрицы, $D_2=\diag(\mu_1, \mu_1^q, \ldots, \mu_1^{q^{n_2-1}})$, $D'_1=\diag(\lambda_2, \lambda_2^q, \ldots, \lambda_2^{q^{n_1-1}})$. Тогда
\begin{itemize}
  \item[{\em (1)}] Если $t_1u_2=u_2t_1$, то $\mu_1^2=1;$
  \item[{\em (2)}] Если $t_1u_2=-u_2t_1$, то $\mu_1^2=\lambda_2^2=-1$, $n_1$~--- нечетно. Более того, если $n_1>1$, то $\lambda_2^{q-1}=-1$.
\end{itemize}
\end{lem}
\begin{proof}
Рассуждения аналогичны доказательству леммы~\ref{lem2}.

\noindent (1) Пусть $D=\bd(D_2,(D_2)^{-1})$, $D'=\bd(D'_1,(D'_1)^{-1})$. Равенство $t_1u_2=u_2t_1$ влечет $D^{\tau_2}=D$. Следовательно, $D_2$ и $D_2^{-1}$ совпадают, откуда $\mu_1=\mu_1^{-1}$ и $\mu_1^2=1$. \\
(2) В данном случае равенство $t_1u_2=-u_2t_1$ равносильно двум равенствам $D^{\tau_2}=-D$ и $(D')^{\varpi_1}=-(D')$. Из первого равенства следует, что $D_2^{-1}=-D_2$ и $\mu_1^2=-1$. Расписывая второе равенство поэлементно, получаем:
\[\left\{\begin{array}{rcl}
\lambda_2 & = & -\lambda_2^{-q^{n_1-1}} \\
\lambda_2^q & = & -\lambda_2 \\
& \vdots &  \\
\lambda_2^{q^{n_1-1}} & = & -\lambda_2^{q^{n_1-2}} \\
\lambda_2^{-1} & = & -\lambda_2^{q^{n_1-1}} \\
\lambda_2^{-q} & = & -\lambda_2^{-1} \\
& \vdots &  \\
\lambda_2^{-q^{n_1-1}} & = & -\lambda_2^{-q^{n_1-2}}
\end{array}
\right.\]
Следовательно, при $n_1>1$ имеем $\lambda_2^{q-1}=-1$ и $\lambda_2^{-1}=-\lambda_2^{q^{n_1-1}}=(-1)^2\lambda_2^{q^{n_1-2}}=\ldots=(-1)^{n_1}\lambda_2$, откуда $\lambda_2^2=(-1)^{n_1}$. Так как $(q-1)$~--- четно, то $n_1$ должно быть нечетным и $\lambda_2^2=-1$. В случае $n_1=1$ получаем, что $\lambda_2^{-1}=-\lambda_2$, откуда $\lambda_2^2=-1$.
\end{proof}

\begin{lem}\label{lem3} 
Пусть $\{n_1, n_2,\ldots, n_m\}$~--- разбиение числа $n$, где $m\geqslant3$.\\ Пусть $s_1=\bd(T_1, T_2,\ldots,T_m)\sigma_1$, $s_2=\bd(T'_1, T'_2,\ldots,T'_m)\sigma_2$, где $T_2=\diag(\mu_1, \mu_2,\ldots, \mu_{n_2})$, $T'_1=\diag(\lambda'_1, \lambda'_2,\ldots, \lambda'_{n_1})$.
Тогда если $s_1s_2=s_2s_1(zI)$, то $z=1$, $\lambda'_1=\lambda'_2=\ldots=\lambda'_{n_1}$, $\mu_1=\mu_2=\ldots=\mu_{n_2}$.
\end{lem}
\begin{proof}
В силу равенства $s_1s_2=s_2s_1(zI)$ имеем $T_mT'_m=T'_mT_mzI_{m}$, откуда следует, что $z=1$. Остальные равенства доказываются аналогично рассуждениям в лемме~\ref{lem1}.
\end{proof}

\begin{lem}\label{lem4}
Пусть $\{n_1, n_2\}$~--- разбиение числа $n$. Пусть $s=\bd(T_1, T_2)\sigma_1$, где $T_1=\diag(\lambda_1, \lambda_2,\ldots, \lambda_{n_1})$, $T_2=\diag(\mu_1, \mu_2,\ldots, \mu_{n_2})$. Тогда $s^{n_1}=zI$ в том и только в том случае, если $\lambda_1\lambda_2\ldots\lambda_{n_1}=\mu_1^{n_1}=\mu_2^{n_1}=\ldots=\mu_{n_2}^{n_1}=z$.
\end{lem}
\begin{proof}

Поскольку $\sigma_1^{n_1}=I_{1}$ получаем, что $$(T_1\sigma_1)^{n_1}=T_1T_1^{\sigma_1^{-1}}T_1^{\sigma_1^{-2}}\ldots T_1^{\sigma_1^{-(n_1-1)}}\sigma^{n_1}=T_1T_1^{\sigma_1^{n_1-1}}T_1^{\sigma_1^{n_2-2}}\ldots T_1^{\sigma_1}= T_1T_1^{\sigma_1}T_1^{\sigma_1^{2}}\ldots T_1^{\sigma_1^{(n_1-1)}}.$$ Полученная матрица будет скалярной поскольку любой ее диагональный элемент будет равен произведению всех диагональных элементов матрицы $T_1$, то есть $(T_1\sigma_1)^{n_1}=\alpha I_{1}$, где $\alpha=\lambda_1\lambda_2\ldots\lambda_{n_1}$. Остальные равенства очевидны.
\end{proof}

\section{Доказательство теоремы \ref{th2}}
Напомним, что случай четной характеристики был рассмотрен в замечании~\ref{rem01}, поэтому всюду далее рассматривается случай нечетной характеристики поля и тогда центр группы $\Sp_{2n}(q)$ состоит из матриц $\pm I$. Максимальный тор из предложения~\ref{prop} будем называть максимальным тором, имеющим тип $(\overline{n_1})\ldots(\overline{n_k})(n_{k+1})\ldots(n_m)$. Следующая лемма ограничивает рассмотрение случаев, в которых нормализатор тора может быть расщепляем.

\begin{lem}\label{lem5}
Пусть $T$~--- максимальный тор в группе $G=\Sp_{2n}(q)$, имеющий тип $(\overline{n_1})\ldots(\overline{n_k})(n_{k+1})\ldots(n_m)$, и $\widetilde{T}$~--- образ тора $T$ в $\widetilde{G}=\PSp_{2n}(q)$. Тогда
\begin{itemize}
  \item[{\em (1)}] Тор $T$ не имеет дополнения в $N$;
  \item[{\em (2)}] Если $m\geqslant 3$, то тор $\widetilde{T}$ не имеет дополнения в $N_{\widetilde{G}}(\widetilde{T})$.
\end{itemize}
\begin{proof} (1) Предположим противное. Пусть $H$~--- дополнение для $T$ в $N$. Поскольку $(\varpi_i)^{n_i}=\tau_i$, то в любом случае элемент $\tau_1$ лежит в $C_W(w)\simeq N_G(T)/T$.  Пусть $u_1$~--- прообраз элемента $\tau_1$. Тогда элемент $u_1$ имеет вид
\begin{center}
$u_1=\bd(D_1, D_2,\ldots,D_m, D_1^{-1}(-I_1), D_2^{-1},\ldots,D_m^{-1})\tau_1$,
\end{center}

\noindent для некоторых диагональных матриц $D_1,\ldots,D_m$. Тогда
\begin{center}
$u_1^2=\bd(-I_1, D_2^2,\ldots,D_m^2, -I_1, D_2^{-2},\ldots,D_m^{-2})$.
\end{center}
С другой стороны, если $H$~--- дополнение для $T$ в $N$, то $u_1^2=I$. Противоречие. \\
(2) Предположим противное. Пусть $\widetilde{H}$~--- дополнение для $\widetilde{T}$ в $N_{\widetilde{G}}(\widetilde{T})$, $H$~--- прообраз $\widetilde{H}$ в $N$. Поскольку $(\varpi_i)^{n_i}=\tau_i$, то в любом случае элементы $\tau_1,\tau_2$ лежат в $C_W(w)\simeq N_{\widetilde{G}}(\widetilde{T})/\widetilde{T}$. Пусть $u_1, u_2$~--- прообразы элементов $\tau_1,\tau_2$ в $H$. Тогда

\vspace{1em}
\hspace{1cm}$u_1=\bd(D_1, D_2,\ldots,D_m, D_1^{-1}(-I_1), D_2^{-1},\ldots,D_m^{-1})\tau_1$,

\hspace{1cm}$u_2=\bd(D'_1, D'_2,\ldots,D'_m, (D'_1)^{-1}, (D'_2)^{-1}(-I_2),\ldots,(D'_m)^{-1})\tau_2$,

\vspace{1em}
\noindent где $D_2=\diag(\mu_1, \mu_1^q, \ldots, \mu_1^{q^{n_2-1}})$, $D'_1=\diag(\lambda_2, \lambda_2^q, \ldots, \lambda_2^{q^{n_1-1}})$. Поскольку $\widetilde{H}$~--- дополнение для $\widetilde{T}$ в $N_{\widetilde{G}}(\widetilde{T})$, то $u_1u_2=\varepsilon u_2u_1$ $u_1^2=\varepsilon_1I, u_2^2=\varepsilon_2I$ . По условию $m\geqslant 3$, поэтому по лемме~\ref{lem3} получаем $u_1u_2=u_2u_1$, и далее в силу замечания~\ref{rem2} и следствия~\ref{cor1} имеем $\mu_1^2=1$. Так как $u_1^{2}=\varepsilon_1I_n$, то по лемме~\ref{lem4} получаем $-1=\mu_1^{2}$. Противоречие с $\mu_1^2=1$.
\end{proof}
\end{lem}

\begin{lem}\label{lem6}
Пусть $T$~--- максимальный тор в группе $G=\Sp_{2n}(q)$, имеющий тип $(n)$ или $(\overline{n})$, и $\widetilde{T}$~--- образ тора $T$ в $\widetilde{G}=\PSp_{2n}(q)$. Тогда тор $\widetilde{T}$ имеет дополнение в $N_{\widetilde{G}}(\widetilde{T})$.
\begin{proof}
Пусть $t_1=\bd(I_1,C_1)\varpi_1$, $s_1=\omega_1$, $u_1=\bd(I_1,-I_1)\tau_1$. Тогда $t_1, s_1, u_1$ принадлежат $\Sp_{2n}(q)$, $t_1^{2n}=-I$, $s_1^n=I$, $u_1^2=-I$, $s_1u_1=u_1s_1$. Для  тора типа $(\overline{n})$ положим $H=\langle t_1\rangle$, а для тора $(n)$ положим $H=\langle s_1\rangle\times\langle u_1\rangle$. Пусть $\widetilde{H}$~--- образ группы $H$ в $\PSp_{2n}(q)$, тогда $\widetilde{H}$~--- дополнение для тора $\widetilde{T}$ в $N_{\widetilde{G}}(\widetilde{T})$.
\end{proof}
\end{lem}

\begin{lem}\label{lem7}
Пусть $T$~--- максимальный тор в группе $G=\Sp_{4}(q)$, $\widetilde{T}$~--- образ тора $T$ в $\widetilde{G}=\PSp_{4}(q)$. Тогда
\begin{itemize}
  \item[{\em (1)}] Если тор $T$ имеет тип $(\overline{1})(1)$, то $\widetilde{T}$ не имеет дополнения в $N_{\widetilde{G}}(\widetilde{T})$;
  \item[{\em (2)}] Если тор $T$ имеет тип $(1)(1)$, то $\widetilde{T}$ имеет дополнение в $N_{\widetilde{G}}(\widetilde{T})$ тогда и только тогда, когда $q\equiv1\pmod4$;
  \item[{\em (3)}] Если тор $T$ имеет тип $(\overline{1})(\overline{1})$, то $\widetilde{T}$ имеет дополнение в $N_{\widetilde{G}}(\widetilde{T})$ тогда и только тогда, когда $q\equiv3\pmod4$;
\end{itemize}
\begin{proof}
(1) Предположим противное. Пусть $\widetilde{H}$~--- дополнение для $\widetilde{T}$ в $N_{\widetilde{G}}(\widetilde{T})$, $H$~--- прообраз $\widetilde{H}$ в $N$. Пусть $t_1, u_2$~--- прообразы элементов  $\varpi_1,\tau_2$ в $H$. Тогда

\begin{center}
$t_1=\diag(\lambda_1, \mu_1, -\lambda_1^{-1}, \mu_1^{-1})\varpi_1,\qquad u_2=\diag(\lambda_2, \mu_2, \lambda_2^{-1}, -\mu_2^{-1})\tau_2$,
\end{center}

\noindent где $\lambda_1^{q+1}=\lambda_2^{q+1}=1$, $\mu_1^{q-1}=\mu_2^{q-1}=1$. Поскольку $\widetilde{H}$~--- дополнение для $\widetilde{T}$ в $N_{\widetilde{G}}(\widetilde{T})$, то $t_1^2=\varepsilon_1I, u_2^2=\varepsilon_2I$ . С другой стороны, $t_1^2=\bd(-1, \mu_1^2, -1, \mu_1^{-2})$ и $u_2^2=\bd(\lambda_2^2,-1,\lambda_2^{-2},-1)$. Следовательно, $\mu_1^2=\lambda_2^2=-1$. Таким образом, $\mu_1^2=-1$ и $\mu_1^{q-1}=1$, что возможно только при $q\equiv1\pmod4$, а также $\lambda_2^2=-1$ и $\lambda_2^{q+1}=1$, что возможно только при $q\equiv3\pmod4$. Противоречие.

\noindent(2),(3) Необходимость условий в пунктах (2) и (3) доказывается аналогично пункту (1). Положим
\[ s_1=\lambda\left(  \begin{array}{cccc}
                0 & 0 & 1 & 0\\
                0 & 1 & 0 & 0 \\
                1 & 0 & 0 & 0 \\
                0 & 0 & 0 & -1 \\
           \end{array} \right),
s_2=\lambda\left(  \begin{array}{cccc}
                1 & 0 & 0 & 0\\
                0 & 0 & 0 & 1 \\
                0 & 0 & -1 & 0 \\
                0 & 1 & 0 & 0 \\
           \end{array} \right),
w=\left(  \begin{array}{cccc}
                0 & 1 & 0 & 0\\
                1 & 0 & 0 & 0 \\
                0 & 0 & 0 & 1 \\
                0 & 0 & 1 & 0 \\
           \end{array} \right).\]

В случае (2) возьмем $\lambda\in\overline{\F}_p$, такой что $\lambda^2=-1$ и $\lambda^{q-1}=1$, а в случае (3) возьмем $\lambda\in\overline{\F}_p$, такой что $\lambda^2=-1$ и $\lambda^{q+1}=1$. Тогда $s_1, s_2, w\in N, s_1^2=s_2^2=-I, w^2=I$. Кроме того, $s_1s_2=-s_2s_1, s_1^w=s_2$. Пусть $H=\langle s_1,s_2,w\rangle$, $\widetilde{H}$~--- образ группы $H$ в $\PSp_{2n}(q)$, тогда $\widetilde{H}$~--- дополнение для $\widetilde{T}$ в $N_{\widetilde{G}}(\widetilde{T})$.
\end{proof}
\end{lem}

\begin{lem}\label{lem8}
Пусть $T$~--- максимальный тор в группе $G=\Sp_{2n}(q)$, имеющий тип $(\overline{n_1})(\overline{n_2})$ и $\widetilde{T}$~--- образ тора $T$ в $\widetilde{G}=\PSp_{2n}(q)$. Тогда $\widetilde{T}$ имеет дополнение в $N_{\widetilde{G}}(\widetilde{T})$ в том и только в том случае, если $n_1, n_2$~--- нечетные и $q\equiv3\pmod4$.
\begin{proof}
Сначала докажем необходимость. Пусть $\widetilde{H}$~--- дополнение для $\widetilde{T}$ в $N_{\widetilde{G}}(\widetilde{T})$, $H$~--- прообраз $\widetilde{H}$ в $N$. Пусть $t_1, t_2$~--- прообразы элементов $\varpi_1,\varpi_2$ в $H$. Тогда

\begin{center}
$t_1=\bd(D_1, D_2,D_1^{-1}C_1, D_2^{-1})\varpi_1$, $t_2=\bd(D'_1, D'_2, (D'_1)^{-1}, (D'_2)^{-1}C_2)\varpi_2$,
\end{center}

\noindent где $D_2=\diag(\mu_1, \mu_1^q, \ldots, \mu_1^{q^{n_2-1}})$, $D'_1=\diag(\lambda_2, \lambda_2^q, \ldots, \lambda_2^{q^{n_1-1}})$.
Поскольку $\widetilde{H}$~--- дополнение для $\widetilde{T}$ в $N_{\widetilde{G}}(\widetilde{T})$, то $t_1t_2=\varepsilon t_2t_1$, $(t_1)^{2n_1}=\varepsilon_1I, (t_2)^{2n_2}=\varepsilon_2I$. Тогда из равенств $(t_1)^{2n_1}=\varepsilon_1I$ и $(t_2)^{2n_2}=\varepsilon_2I$ в силу леммы~\ref{lem4} имеем $\mu_1^{2n_1}=-1$ и $\lambda_2^{2n_2}=-1$ соответственно. Случай $t_1t_2= t_2t_1$ невозможен в силу пункта (1) леммы~\ref{lem2}, поэтому применяя пункт (2) леммы~\ref{lem2} получаем, что $n_1, n_2$~--- нечетны. Более того, $\mu_1^{q-1}=-1$, если $n_2>1$ и $\lambda_2^{q-1}=-1$, если $n_1>1$. Случай $n_1=n_2=1$ разобран в лемме~\ref{lem7}, поэтому $\mu_1^{q-1}=-1$ или $\lambda_2^{q-1}=-1$. Поскольку $\mu_1^{2n_1}=-1$ и $\lambda_2^{2n_2}=-1$, то $q\equiv3\pmod4$. Необходимость доказана, покажем достаточность. Пусть $n_1, n_2$~--- нечетные и $q\equiv3\pmod4$. Возьмем $\lambda\in\overline{\F}_p$, такой что $\lambda^2=-1$, тогда $\lambda^q=-\lambda$. Определим
$$t_1=\bd(I_1, D_2,C_1, D_2^{-1})\varpi_1, t_2=\bd(D'_1, I_2, (D'_1)^{-1}, C_2)\varpi_2,$$
где $D_2=\diag(\lambda,-\lambda,\ldots,\lambda,-\lambda,\lambda)$, $D'_1=\diag(\lambda,-\lambda,\ldots,\lambda,-\lambda,\lambda)$.
Тогда $t_1, t_2$ лежат в $\Sp_{2n}(q), t_1^{2n_1}=-I, t_2^{2n_2}=-I$ (см. лемму~\ref{lem4}) и $t_1t_2=-t_2t_1$. Если $n_1\neq n_2$, то $C_W(w)=\langle \varpi_1,\varpi_2\rangle\simeq\mathbb{Z}_{2n_1}\times\mathbb{Z}_{2n_2}$ и в качестве $H$ можно взять $H=\langle t_1, t_2\rangle$. Пусть $\widetilde{H}$~--- образ группы $H$ в $\PSp_{2n}(q)$, тогда $\widetilde{H}$~--- дополнение для $\widetilde{T}$ в $N_{\widetilde{G}}(\widetilde{T})$. Если $n_1=n_2$, то $C_W(w)\simeq(\mathbb{Z}_{n}\times\mathbb{Z}_{n})\rtimes\mathbb{Z}_2$. В этом случае пусть
\[\omega=\left(  \begin{array}{cccc}
                0 & I_1 & 0 & 0\\
                I_1 & 0 & 0 & 0 \\
                0 & 0 & 0 & I_1 \\
                0 & 0 & I_1 & 0 \\
           \end{array} \right).\]
Тогда $\omega\in\Sp_{2n}(q), \omega^2=I, t_1^{\omega}=t_2$. Положим $H=\langle t_1,t_2,\omega\rangle$, $\widetilde{H}$~--- образ группы $H$ в $\PSp_{2n}(q)$. Тогда $\widetilde{H}$~--- дополнение для $\widetilde{T}$ в $N_{\widetilde{G}}(\widetilde{T})$.
\end{proof}
\end{lem}

\begin{lem}\label{lem9}
Пусть $T$~--- максимальный тор в группе $G=\Sp_{2n}(q)$, имеющий тип $(n_1)(n_2)$ и $\widetilde{T}$~--- образ тора $T$ в $\widetilde{G}=\PSp_{2n}(q)$. Тогда $\widetilde{T}$ не имеет дополнения в $N_{\widetilde{G}}(\widetilde{T})$ в том и только в том случае, если $n_1, n_2$~--- нечетные и $q\equiv3\pmod4$.
\begin{proof} Рассмотрим сначала случай $q\equiv1\pmod4$. Пусть $\lambda\in\overline{\F}_p$, такой что $\lambda^2=-1, \lambda^{q-1}=1$. Определим следующие элементы
\begin{center}
$s_1=\omega_1, s_2=\omega_2$, $u_1=\bd(I_1, \lambda I_2,-I_1, \lambda^{-1}I_2)\tau_1, u_2=\bd(\lambda I_1, I_2, \lambda^{-1}I_1, -I_2)\tau_2.$
\end{center}
Тогда $s_i,u_i\in\Sp_{2n}(q)$, $s_i^{n_i}=I, u_i^2=-I, s_i^{u_i}=s_i, s_1s_2=s_2s_1, u_1u_2=-u_2u_1, s_1u_2=u_2s_1, s_2u_1=u_1s_2$, где $i=1,2$.
Если $n_1\neq n_2$, то $C_W(w)=\langle\omega_1,\tau_1,\omega_2,\tau_2\rangle\simeq (\mathbb{Z}_{n_1}\times\mathbb{Z}_2)\times(\mathbb{Z}_{n_2}\times\mathbb{Z}_2)$ и определим $H=\langle s_1, u_1,s_2, u_2\rangle$. Тогда  образ $\widetilde{H}$ группы $H$ в $\PSp_{2n}(q)$ будет дополнением для $\widetilde{T}$ в $N_{\widetilde{G}}(\widetilde{T})$. Если $n_1=n_2$, то $C_W(w)\simeq((\mathbb{Z}_{n_1}\times\mathbb{Z}_2)\times(\mathbb{Z}_{n_2}\times\mathbb{Z}_2))\rtimes\mathbb{Z}_2$. Пусть $\omega$~--- элемент, определенный в доказательстве леммы~\ref{lem8}, тогда $\omega\in\Sp_{2n}(q), \omega^2=I, s_1^{\omega}=s_2, u_1^{\omega}=u_2$. Положим $H=\langle s_1, u_1,s_2, u_2,\omega\rangle$, $\widetilde{H}$~--- образ $H$ в $\PSp_{2n}(q)$. Тогда $\widetilde{H}$~--- дополнение для $\widetilde{T}$ в $N_{\widetilde{G}}(\widetilde{T})$.

Перейдем к рассмотрению случая $q\equiv3\pmod4$. Пусть $\widetilde{H}$~--- дополнение для $\widetilde{T}$ в $N_{\widetilde{G}}(\widetilde{T})$, $H$~--- прообраз $\widetilde{H}$ в $N$. Покажем, что в этом случае $n_1,n_2$ обязаны быть четными. Пусть $s_1,u_1,s_2,u_2$~--- прообразы элементов $\omega_1,\tau_1,\omega_2,\tau_2$ в $H$. Тогда

\vspace{1em}
\hspace{1em}$s_1=\bd(D_1, D_2,D_1^{-1}, D_2^{-1})\omega_1$, $u_1=\bd(B_1, B_2,-B_1^{-1}, B_2^{-1})\tau_1$,

\hspace{1em}$s_2=\bd(D'_1, D'_2, (D'_1)^{-1}, (D'_2)^{-1})\omega_2$, $u_2=\bd(B'_1, B'_2,(B'_1)^{-1}, -(B'_2)^{-1})\tau_2.$

\vspace{1em}
\noindent где $D_1=\diag(\lambda_1, \lambda_1^q,\ldots, \lambda_1^{q^{n_1-1}})$, $D_2=\diag(\mu_1, \mu_1^q, \ldots, \mu_1^{q^{n_2-1}})$, $D'_1=\diag(\lambda_2, \lambda_2^q, \ldots, \lambda_2^{q^{n_1-1}})$, $D'_2=\diag(\mu_2, \mu_2^q, \ldots, \mu_2^{q^{n_2-1}})$, $B_1=\diag(\alpha_1, \alpha_1^q,\ldots, \alpha_1^{q^{n_1-1}})$, $B_2=\diag(\beta_1, \beta_1^q, \ldots, \beta_1^{q^{n_2-1}})$, $B'_1=\diag(\alpha_2, \alpha_2^q, \ldots, \alpha_2^{q^{n_1-1}})$, $B'_2=\diag(\beta_2, \beta_2^q, \ldots, \beta_2^{q^{n_2-1}})$.

Поскольку $\widetilde{H}$~--- дополнение для $\widetilde{T}$ в $N_{\widetilde{G}}(\widetilde{T})$, то, в частности, должны выполняться равенства $u_1^{2}=\varepsilon_1I, u_2^2=\varepsilon_2I$, $s_1u_2=\varepsilon_3u_2s_1$, $s_2u_1=\varepsilon_4u_1s_2$. По лемме~\ref{lem4} из равенств $u_1^{2}=\varepsilon_1I, u_2^2=\varepsilon_2I$ получаем, что $\beta_1^2=-1$, $\alpha_2^2=-1$. Далее, из соотношения $s_1u_2=\varepsilon_3u_2s_1$, в частности, следуют равенства
\[\left\{\begin{array}{rcl}
\alpha_2 & = & \varepsilon_3\alpha_2^{q^{n_1-1}} \\
\alpha_2^q & = & \varepsilon_3\alpha_2 \\
& \vdots &  \\
\alpha_2^{q^{n_1-1}} & = & \varepsilon_3\alpha_2^{q^{n_1-2}}
\end{array}
\right.\]
Если $\varepsilon_3=1$, то $\alpha_2^{q-1}=\varepsilon_3=1$ и при этом $\alpha_2^2=-1$, откуда $q\equiv1\pmod4$. Следовательно, $\varepsilon_3=-1$ и $\alpha_2=-\alpha_2^{q^{n_1-1}}=(-1)^2\alpha_2^{q^{n_1-2}}=\ldots=(-1)^{n_1}\alpha_2$, что возможно только при четном $n_1$. Аналогично, из равенства $s_2u_1=\varepsilon_4u_1s_2$ следует, что $\varepsilon_4=-1$ и $\beta_1=(-1)^{n_2}\beta_1$. Таким образом, числа $n_1, n_2$ должны быть четными.

Далее считаем, что $n_1, n_2$ четны и $q\equiv3\pmod4$. Пусть $\lambda\in\overline{\F}_p$, такой что $\lambda^2=-1$; $\xi_1, \xi_2$~--- первообразные корни из 1 степени $q^{n_1}-1, q^{n_2}-1$ соответственно, $\alpha_1=\xi_1^{-1}$, $\lambda_1=\xi_1^{\frac{q-1}{2}}$, $\beta_2=\xi_2^{-1}$, $\mu_2=\xi_2^{\frac{q-1}{2}}$. Тогда $\lambda^{q-1}=-1$, $\lambda_1^{(\frac{q^{n_1}-1}{q-1})}=~-1$, $\mu_2^{(\frac{q^{n_2}-1}{q-1})}=-1$, $\lambda_1^2\cdot\alpha_1^{q-1}=1$, $\mu_2^2\cdot\beta_2^{q-1}=1$. Определим следующие элементы

\begin{center}
$s_1=\bd(D_1, D_2,D_1^{-1}, D_2^{-1})\omega_1$, $s_2=\bd(D'_1, D'_2, (D'_1)^{-1}, (D'_2)^{-1})\omega_2$,
\end{center}

\noindent где  $D_2=\diag(\lambda, -\lambda, \ldots, \lambda, -\lambda)$, $D'_1=\diag(\lambda, -\lambda, \ldots, \lambda, -\lambda)$,

\[D_1=\begin{cases} I_1 & \text{если }n_1\equiv0\pmod4 \\
\diag(\lambda_1, \lambda_1^q, \ldots, \lambda_1^{q^{n_1-1}}) & \text{если }n_1\equiv2\pmod4
\end{cases},
\]
\[D'_2=\begin{cases} I_2 & \text{если }n_2\equiv0\pmod4 \\
\diag(\mu_2, \mu_2^q, \ldots, \mu_2^{q^{n_2-1}}) & \text{если }n_2\equiv2\pmod4.
\end{cases}\]

Матрицы $D_1$ и $D'_2$ выбраны таким образом, чтобы $s_1^{n_1}$ и $s_2^{n_2}$ получились скалярными. Действительно, если $n_1\equiv0\pmod4$, то $\lambda^{n_1}=1$ и по лемме~\ref{lem4} получаем $s_1^{n_1}=I$, а если $n_1\equiv2\pmod4$, то $\lambda_1\lambda_1^q\ldots\lambda_1^{q^{n_1-1}}=\lambda_1^{(\frac{q^{n_1}-1}{q-1})}=-1$, $\lambda^{n_1}=-1$ и по лемме~\ref{lem4} получаем $s_1^{n_1}=-I$. Аналогично, $s_2^{n_2}=I$ при $n_2\equiv0\pmod4$ и $s_2^{n_2}=-I$ при $n_2\equiv2\pmod4$. Определим элементы

\begin{center}
$u_1=\bd(B_1, B_2,-B_1^{-1}, B_2^{-1})\tau_1$, $u_2=\bd(B'_1, B'_2, (B'_1)^{-1}, -(B'_2)^{-1})\tau_2$,
\end{center}

\noindent где  $B_2=\diag(\lambda, -\lambda, \ldots, \lambda, -\lambda)$, $B'_1=\diag(\lambda, -\lambda, \ldots, \lambda, -\lambda)$,

\[B_1=\begin{cases} I_1 & \text{если }n_1\equiv0\pmod4 \\
\diag(\alpha_1, \alpha_1^q, \ldots, \alpha_1^{q^{n_1-1}}) & \text{если }n_1\equiv2\pmod4
\end{cases},
\]
\[B'_2=\begin{cases} I_2 & \text{если }n_2\equiv0\pmod4 \\
\diag(\beta_2, \beta_2^q, \ldots, \beta_2^{q^{n_2-1}}) & \text{если }n_2\equiv2\pmod4.
\end{cases}\]

Матрицы $B_1$ и $B'_2$ выбраны таким образом, чтобы $s_1u_1=u_1s_1$ и $s_2u_2=u_2s_2$. Действительно, пусть $D=\bd(D_1,(D_1)^{-1})$, $B=\bd(B_1,-(B_1)^{-1})$. Тогда равенство $s_1u_1=u_1s_1$ равносильно $D\omega_1B\tau_1=B\tau_1D\omega_1$, откуда $DB^{\omega_1^{-1}}=BD^{\tau_1}$. В случае $n_1\equiv0\pmod4$ последнее равенство тривиально. В случае $n_1\equiv2\pmod4$  равенство $DB^{\omega_1^{-1}}=BD^{\tau_1}$ равносильно условию $\lambda_1\alpha_1^q=\alpha_1\lambda_1^{-1}$, которое выполнено в силу выбора элементов $\lambda_1$ и $\alpha_1$. Аналогично, в случае $n_2\equiv2\pmod4$ равенство $s_2u_2=u_2s_2$ равносильно условию $\mu_2\beta_2^q=\beta_2\mu_2^{-1}$, которое выполняется. Далее, непосредственно проверяются следующие равенства:
\begin{center}
$u_1^2=u_2^2=-I, s_1u_2=-u_2s_1, s_2u_1=-u_1s_2, u_1u_2=-u_2u_1, s_1s_2=-s_2s_1$.
\end{center}
Если $n_1\neq n_2$, то $C_W(w)\simeq(\mathbb{Z}_{n_1}\times\mathbb{Z}_2)\times(\mathbb{Z}_{n_2}\times\mathbb{Z}_2)$ и определим $H=\langle s_1,u_1,s_2,u_2\rangle$. Пусть $\widetilde{H}$~--- образ группы $H$ в $\PSp_{2n}(q)$, тогда $\widetilde{H}$~--- дополнение для $\widetilde{T}$ в $N_{\widetilde{G}}(\widetilde{T})$. Если $n_1=n_2$, то $W\simeq((\mathbb{Z}_{n_1}\times\mathbb{Z}_2)\times(\mathbb{Z}_{n_2}\times\mathbb{Z}_2))\rtimes\mathbb{Z}_2$. Пусть $\omega$~--- элемент, определенный в доказательстве леммы~\ref{lem8}, тогда $\omega\in\Sp_{2n}(q), \omega^2=I, t_1^{\omega}=t_2$. Положим $H=\langle s_1, u_1,s_2, u_2,\omega\rangle$, $\widetilde{H}$~--- образ группы $H$ в $\PSp_{2n}(q)$. Тогда $\widetilde{H}$~--- дополнение для $\widetilde{T}$ в $N_{\widetilde{G}}(\widetilde{T})$.
\end{proof}
\end{lem}

\begin{lem}\label{lem10}
Пусть $T$~--- максимальный тор в группе $G=\Sp_{2n}(q)$, имеющий тип $(\overline{n_1})(n_2)$ и $\widetilde{T}$~--- образ тора $T$ в $\widetilde{G}=\PSp_{2n}(q)$. Тогда $\widetilde{T}$ имеет дополнение в $N_{\widetilde{G}}(\widetilde{T})$ в том и только в том случае, если $n_1$~--- нечетно, $n_2$~--- четно и $q\equiv3\pmod4$.
\begin{proof}
Сначала докажем необходимость. Пусть $\widetilde{H}$~--- дополнение для $\widetilde{T}$ в $N_{\widetilde{G}}(\widetilde{T})$, $H$~--- прообраз $\widetilde{H}$ в $N$. Пусть $t_1, u_2$~--- прообразы элементов $\varpi_1,\tau_2$ в $H$. Тогда

\begin{center}
$t_1=\bd(D_1, D_2,D_1^{-1}C_1, D_2^{-1})\varpi_1$, $u_2=\bd(D'_1, D'_2, (D'_1)^{-1}, -(D'_2)^{-1})\tau_2$,
\end{center}

\noindent где $D_2=\diag(\mu_1, \mu_1^q, \ldots, \mu_1^{q^{n_2-1}})$, $\mu_1^{q^{n_2}-1}=1$, $D'_1=\diag(\lambda_2, \lambda_2^q, \ldots, \lambda_2^{q^{n_1-1}})$, $\lambda_2^{q^{n_1}+1}=1$.
Поскольку $\widetilde{H}$~--- дополнение для $\widetilde{T}$ в $N_{\widetilde{G}}(\widetilde{T})$, то $t_1u_2=\varepsilon u_2t_1$, $(t_1)^{2n_1}=\varepsilon_1I, u_2^2=\varepsilon_2I$. Тогда из равенств $(t_1)^{2n_1}=\varepsilon_1I$ и $t_2^2=\varepsilon_2I$ в силу леммы~\ref{lem4} имеем $\mu_1^{2n_1}=-1$ и $\lambda_2^2=-1$ соответственно. Случай $t_1u_2= u_2t_1$ невозможен в силу пункта (1) леммы~\ref{lem2.5}, поэтому применяя пункт (2) леммы~\ref{lem2.5} получаем, что $n_1$~--- нечетно. Более того, если $n_1>1$, то по пункту (2) леммы~\ref{lem2.5} имеем $\lambda_2^{q-1}=-1$, откуда $q\equiv3\pmod4$. Если $n_1=1$, то должно выполняться $\lambda_2^{q+1}=1$, откуда $q\equiv3\pmod4$. Таким образом, в любом случае $q\equiv3\pmod4$.

Пусть $s_2$~--- прообраз элемента $\omega_2$ в $H$. Тогда должно выполняться равенство $t_1s_2=\varepsilon_3 s_2t_1$, откуда, в частности, $D_2^{\omega_2}=\varepsilon_3D_2$. Следовательно,
\[\left\{\begin{array}{rcl}
\mu_1 & = & \varepsilon_3\mu_1^{q^{n_2-1}} \\
\mu_1^q & = & \varepsilon_3\mu_1 \\
& \vdots &  \\
\mu_1^{q^{n_2-1}} & = & \varepsilon_3\mu_1^{q^{n_2-2}}
\end{array}
\right.\]
Если $\varepsilon_3=1$, то $\mu_1^{q-1}=\varepsilon_3=1$ и при этом $\mu_1^2=-1$, откуда $q\equiv1\pmod4$, что невозможно. Следовательно, $\varepsilon_3=-1$ и $\mu_1=-\mu_1^{q^{n_2-1}}=(-1)^2\mu_1^{q^{n_2-2}}=\ldots=(-1)^{n_2}\mu_1$, откуда следует, что $n_2$ обязано быть четным. Необходимость доказана, покажем достаточность.

Далее считаем, что $n_1$ нечетно, $n_2$ четно и $q\equiv3\pmod4$. Пусть $\lambda\in\overline{\F}_p$, такой что $\lambda^2=-1$; $\xi_2$~--- первообразный корень из 1 степени $q^{n_2}-1$, $\beta_2=\xi_2^{-1}$, $\mu_2=\xi_2^{\frac{q-1}{2}}$. Тогда $\lambda^{q-1}=-1$, $\mu_2^{(\frac{q^{n_2}-1}{q-1})}=-1$, $\mu_2^2\cdot\beta_2^{q-1}=1$. Определим следующие элементы

\begin{center}
$t_1=\bd(I_1, D_2,C_1, D_2^{-1})\varpi_1$,\quad $s_2=\bd(D'_1, D'_2, (D'_1)^{-1}, (D'_2)^{-1})\omega_2$,
\end{center}

\noindent где  $D_2=\diag(\lambda, -\lambda, \ldots, \lambda, -\lambda)$, $D'_1=\diag(\lambda,-\lambda,\ldots,\lambda,-\lambda,\lambda)$,

\[D'_2=\begin{cases} I_2 & \text{если }n_2\equiv0\pmod4 \\
\diag(\mu_2, \mu_2^q, \ldots, \mu_2^{q^{n_2-1}}) & \text{если }n_2\equiv2\pmod4.
\end{cases}\]
Матрица $D'_2$ выбрана таким образом, чтобы $s_2^{n_2}$ получилась скалярной. Действительно, если $n_2\equiv0\pmod4$, то $\lambda^{n_2}=1$ и по лемме~\ref{lem4} получаем $s_2^{n_2}=I$, а если $n_2\equiv2\pmod4$, то $\mu_2\mu_2^q\ldots\mu_2^{q^{n_2-1}}=\mu_2^{(\frac{q^{n_2}-1}{q-1})}=-1$, $\lambda^{n_1}=-1$ и по лемее~\ref{lem4} получаем $s_1^{n_1}=-I$. Наконец, определим элемент
\begin{center}
$u_2=\bd(B_1, B_2, (B_1)^{-1}, -(B_2)^{-1})\tau_2$,\quad где $B_1=\diag(\lambda,-\lambda,\ldots,\lambda,-\lambda,\lambda)$,
\end{center}
\[B_2=\begin{cases} I_2 & \text{если }n_2\equiv0\pmod4 \\
\diag(\beta_2,\beta_2^q, \ldots,\beta_2^{q^{n_1-1}}) & \text{если }n_2\equiv2\pmod4
\end{cases},
\]

Матрица $B_2$ выбрана таким образом, чтобы $s_2u_2=u_2s_2$. Действительно, пусть $D=\bd(D'_2,(D'_2)^{-1})$, $B=\bd(B_2,-(B_2)^{-1})$. Тогда равенство $s_2u_2=u_2s_2$ равносильно $D\omega_2B\tau_2=B\tau_2D\omega_2$, откуда $DB^{\omega_2^{-1}}=BD^{\tau_2}$. В случае $n_2\equiv0\pmod4$ последнее равенство тривиально. В случае $n_2\equiv2\pmod4$  равенство $DB^{\omega_2^{-1}}=BD^{\tau_2}$ равносильно условию $\mu_2\beta_2^q=\beta_2\mu_2^{-1}$, которое выполнено в силу выбора элементов $\mu_2$ и $\beta_2$. Далее, непосредственно проверяются следующие равенства:
\begin{center}
$t_1^{2n_1}=u_2^2=-I, t_1u_2=-u_2t_1, t_1s_2=-s_2t_1$.
\end{center}
 Поскольку $n_1\neq n_2$, то $C_W(w)\simeq\mathbb{Z}_{2n_1}\times(\mathbb{Z}_{n_2}\times\mathbb{Z}_2)$. Определим $H=\langle t_1,s_2,u_2\rangle$, $\widetilde{H}$~--- образ группы $H$ в $\PSp_{2n}(q)$. Тогда $\widetilde{H}$~--- дополнение для $\widetilde{T}$ в $N_{\widetilde{G}}(\widetilde{T})$.
\end{proof}
\end{lem}

\noindent В леммах~\ref{lem5}--\ref{lem10} и замечании~\ref{rem01} рассмотрены все типы максимальных торов в группе $\PSp_{2n}(q)$, откуда следует теорема~\ref{th2}. Следствие~\ref{cor2.5} следует из замечания~\ref{rem01} и пункта (1) леммы~\ref{lem5}.\\

В заключение автор выражает благодарность А.А. Бутурлакину и Е.П. Вдовину за обсуждение работы и ценные замечания.

\noindent\textit{Гальт Алексей Альбертович}

\noindent\textit{Математическтй факультет,}

\noindent\textit{Университет науки и технологий Китая,}

\noindent\textit{Хэфэй 230026, Китай}

\noindent\verb"galt84@gmail.com"

\end{document}